\documentclass[12pt]{amsart}
\usepackage{pb-diagram}

\usepackage{amscd}
\usepackage{amsfonts}
\usepackage{amsmath}
\usepackage{amssymb}
\usepackage{color}

\newcommand{\ric}{\textnormal{Ric}}

\newcommand{\ep}{\varepsilon}

\newcommand{\var}{\textnormal{Var}}

\newcommand{\cH}{\mathcal{H}}

\newcommand{\dbar}{\overline{\partial}}

\newcommand{\ddt}[1]{\frac{\partial #1}{\partial t}}
\newcommand{\dds}[1]{\frac{\partial #1}{\partial s}}

\newcommand{\ddbar}{\sqrt{-1}\partial\dbar}

\newcommand{\cK}{\mathcal{K}}
\newcommand{\cL}{\mathcal{L}}

\newcommand{\cN}{\mathcal{N}}
\newcommand{\cW}{\mathcal{W}}

\newcommand{\vol}{\textnormal{Vol}}

\newcommand{\diam}{\textnormal{Diam}}

\newtheorem{theorem}{Theorem}[section]
\newtheorem{lemma}[theorem]{Lemma}

\newtheorem{proposition}[theorem]{Proposition}

\newtheorem{claim}[theorem]{Claim}

\numberwithin{equation}{section}

\theoremstyle{definition}
\newtheorem{remark}[theorem]{Remark}

\theoremstyle{definition}
\newtheorem{definition}[theorem]{Definition}

\begin{document}

\title{A new proof of Perelman's scalar curvature and diameter estimates for the K\"ahler-Ricci flow on Fano manifolds}

\author{Wangjian Jian$^*$, Jian Song$^\dagger$ and Gang Tian$^\ddagger$}

\thanks{Wangjian Jian is supported in part by NSFC No.12201610, NSFC No.12288201, MOST No.2021YFA1003100. Jian Song is supported in part by National Science Foundation grant DMS-2203607. Gang Tian is supported in part by NSFC No.11890660, MOST No.2020YFA0712800.}

\address{$^*$ Institute of Mathematics, Academy of Mathematics and Systems Science, Chinese Academy of Sciences, Beijing, 100190, China}

\email{wangjian@amss.ac.cn}

\address{$^\dagger$ Department of Mathematics, Rutgers University, Piscataway, NJ 08854}

\email{jiansong@math.rutgers.edu}

\address{$^\ddagger$ BICMR and SMS, Peking University, Beijing 100871, China}

\email{gtian@math.pku.edu.cn}

\maketitle

\begin{abstract}  In this note, we give a new proof for Perelman's scalar curvature and diameter estimates for the K\"ahler-Ricci flow on Fano manifolds. The proof relies on a new Harnack estimate for a special family of functions in space-time. Our new approach initiates the work in \cite{JST23a} for general finite time solutions of the K\"ahler-Ricci flow.

\end{abstract}

%\tableofcontents
%%%%%%%%%%%%%%%%%%%%%%%%%%%%%%%%%%%%%%%%%%%%%%%%%%%%%%%%%%%%%%%%%%%%%%%%%%%%%%%%%%%%

\bigskip

\section{Introduction}

Let $X$ be an $n$-dimensional  K\"ahler manifold with $c_1(X)>0$. The Ricci flow introduced by Hamilton \cite{Ham} is a canonical deformation of Riemannian metrics and it preserves the K\"ahler structure on $X$ if the initial metric is K\"ahler.  We consider the unnormalized K\"ahler-Ricci flow
\begin{equation}\label{unkrflow1}
\left\{
\begin{array}{l}
{ \displaystyle \ddt{g(t)} = -\ric(g(t)) ,}\\
\\
g(0)=g_0 \in c_1(X), 
\end{array} \right. 
\end{equation}
on   $X$ with an initial K\"ahler metric $g_0$. The flow must develop finite time singularities at $T=1$ with limiting cohomology class $\lim_{t\rightarrow 1} [g(t)]=0$. 

Our main result of the paper is the following Harnack estimate along the K\"ahler-Ricci flow (\ref{unkrflow1}).

\begin{theorem}\label{main1}
For any $B>0$, there exists   $C=C(n, X, g_0, B)>0$, such that  for any $\mathbf{t} \in [1/2, 1)$ and a positive function $v \in C^{1}\left(X\times [\mathbf{t}-(1-\mathbf{t}), \mathbf{t}] \right)$ satisfying  
\begin{equation}\label{main1LY}
\frac{ |\partial_t v |}{ v } + \frac{|\nabla v |^2}{ v^2 } \leq \frac{B}{1-t},
\end{equation}
on $X\times [\mathbf{t} - (1-\mathbf{t} ), \mathbf{t}]$, we have
\begin{equation}\label{main1har}
\frac{ \sup_{X} v ( \cdot, \mathbf{t} ) }{ \inf_{X} v ( \cdot, \mathbf{t} ) } \leq C .
\end{equation}
\end{theorem}
Theorem \ref{main1} provides a new proof of Perelman's fundamental estimates for the finite time solution of the K\"ahler-Ricci flow (\ref{unkrflow1}) on the Fano manifold $X$ (cf. \cite{SeT}).
\begin{theorem}[Perelman]\label{eoper} 
Let $g(t)$ be the maximal solution of (\ref{unkrflow1}) on the Fano manifold $X$ for $t\in [0, 1)$. There exists $C=C(n, X, g_0)>0$ such that
\begin{enumerate}
\medskip
\item $R(x, t) \leq C(1-t)^{-1}$, 
\medskip
\item $\diam(X,  g(t))\leq C(1-t)^{1/2}$,
%\medskip
%\item $|u|_{C^1}\leq C$, on $X\times [0, \infty)$,
%
\medskip
\end{enumerate}
on $X\times [0, 1)$, where $R(\cdot, t)$ is the scalar curvature of $g(t)$ and $\diam(X, g(t))$ is the diameter of $(X, g(t))$ .
\end{theorem}

After applying the rescaling 
\begin{equation}\label{rb unkrf and nkrf}
s=-\ln (1-t),~~ t=1-e^{-s},~~ \tilde{g}(s)=(1-t)^{-1}g(t),~~t\in [0,1), 
\end{equation}
we have the following Fano K\"ahler-Ricci flow as the normalized flow of (\ref{unkrflow1})
\begin{equation}\label{nkrflow1}
\left\{
\begin{array}{l}
{ \displaystyle \dds{\tilde{g}(s)} = -\ric(\tilde{g}(s)) + \tilde{g}(s),}\\
\\
\tilde{g}(0) =g_0.
\end{array} \right.
\end{equation}
The Fano K\"ahler-Ricci flow (\ref{nkrflow1}) has a long-time solution for $s\in [0, \infty)$.  The scalar curvature and diameter of $\tilde g(s)$ are both uniformly bounded on $X\times [0, \infty)$  following Perelman's estimates (Theorem \ref{eoper}).  The Hamilton-Tian conjecture predicts that $(X, \tilde g(s))$  subsequentially converges in Gromov-Hausdorff distance to a compact K\"ahler-Ricci soliton with smooth convergence  outside a close singular set of codimension $\geq 4$. Perelman's result laid the foundation for the Hamilton-Tian conjecture, which is solve in \cite{TiZZL} for complex dimension $3$ and in  \cite{CW3, Bam18} for all dimensions. We also refer readers to \cite{TiZXH2, PS, TiZXH1}  among the large body of literatures in the field of the Fano K\"ahler-Ricci flow. 

We summarize our proof for Perelman's estimates (Theorem \ref{eoper}).
\begin{enumerate}
\smallskip
\item We apply a new normalization for the Ricci potential $u=u(x, s)$ under the normalized flow (\ref{nkrflow1}) so that $\inf_X u(\cdot, s)=1$ for all $s>0$. %
\medskip
\item We obtain   the following Li-Yau type's estimates (see \cite{LY}) similar to the original estimates of Perelman
$$
\frac{ |\Delta u|}{ u } + \frac{ |\nabla u|^2}{ u } \leq C,
$$
on $X\times [0, \infty)$.  Theorem \ref{eoper} is then reduced to the uniform upper bound of $u$ by the above estimate and Perelman's $\kappa$-noncollapsing \cite{Per1}.
\medskip
\item  In order to obtain a uniform upper bound for $u$,  Perelman has to employ a geometric contraction argument to establish the uniform diameter bound. Our approach is entirely different and opposite, by establishing a direct bound for $u$ from Theorem \ref{main1} without controlling the diameter. This is achieved by combining Perelman's reduced distance and Wasserstein distance in recent work \cite{Bam20a}. 
 
 \medskip
\end{enumerate}
Unfortunately, the original normalization and contraction argument in Perelman's proof cannot be directly applied to or adapted for general finite time solutions of the K\"ahler-Ricci flow when $X$ is not Fano or when the initial K\"ahler class is not proportional to $c_1(X)$. Our new normalization for the Ricci potential and the Harnack estimate in Theorem \ref{main1} lead to a new approach \cite{JST23a} to achieve local scalar curvature and distance estimates for general finite time solutions of the K\"ahler-Ricci flow. These estimates will play an important role in the analytic minimal model program proposed in \cite{ST, ST2, ST4, JST23a}. 

\medskip

\section{The Harnack estimate on Fano K\"ahler-Ricci flow}\label{HeofFanoKRF}

In this section, we prove Theorem \ref{main1}. First, we will give some background for the Ricci flow theory, mainly based on \cite{Bam20a}.

%%%%%%%%%%%%%%%%%%%%%%%%%%%%%%%%%%%%%%%%%%%%%%%%%%%%%%%%%%%%%%%%%%%%%%%%%%%%%%%%%%%%

%\medskip
\subsection{Preliminary results on entropy and heat kernel bounds}

Let $(M, g(t))_{t\in I}$ be a smooth Ricci flow on a compact $n$-dimensional manifold with the interval $I\subset \mathbb{R}$. The heat operator associated to $(M, g(t))$ is defined by
$$\Box = \ddt{} - \Delta$$
and the conjugate heat operator is defined by
$$\Box^* = - \ddt{} - \Delta + R, $$
where $\Delta$ is the Laplacian associated to $g(t)$ and $R$ is the scalar curvature of $g(t)$.

For any $(x, t)$, $(y, s) \in M \times I$ with $s\leq t$, we denote by $K(x,t; y,s)$ the heat kernel of the Ricci flow based at $(y,s)$ satisfying
\begin{equation}
\Box K(\cdot, \cdot; y,s)=0, ~~ \lim_{t\rightarrow s^+} K(\cdot, t, ; y, s) = \delta_y,
\end{equation}
where $\delta_y$ is the Dirac measure at $y$. Similarly, $K(x, t; \cdot,\cdot)$ is the conjugate heat kernel based at $(x,t)$ satisfying
\begin{equation}
\Box^* K(x, t; \cdot, \cdot)=0, ~ \lim_{s\rightarrow t^-} K(x, t; \cdot, s) =\delta_x.
\end{equation}

Using the conjugate heat kernel, we can define the conjugate heat measure $\nu_{x, t; s}$ based at $(x,t)$ by
\begin{equation}
d\nu_{x,t;s} = K(x,t; \cdot, s) dg(t) = (4\pi\tau)^{-n/2} e^{-f} dg(t),
\end{equation}
where $\tau=t-s$ and $f\in C^\infty(M \times (-\infty, t))$ is called the potential of the conjugate heat measure $\nu_{x,t;s}$.

For two probability measures $\mu_1$ and $\mu_2$ on a Riemannian manifold $(M, g)$, the Wasserstein $W_1$-distance between $\mu_1$ and $\mu_2$ is defined by
\begin{equation}
d^g_{W_1}(\mu_1, \mu_2) = \sup_f \left( \int_M f d\mu_1 - \int_M f d\mu_2 \right), 
\end{equation}
where the supremum is taken over all bounded $1$-Lipschitz function on $(M, g)$. The variance between $\mu_1$ and $\mu_2$ is defined by
\begin{equation}
\var(\mu_1, \mu_2) = \int_{(x_1, x_2) \in M \times M}  d_g^2(x_1, x_2) d\mu_1(x) d\mu_2(x_2).
\end{equation}
The following basic relation is proved in \cite{Bam20a} between the Wasserstein $W_1$-distance and the variance
\begin{equation}
d^g_{W_1} (\mu_1, \mu_2) \leq \sqrt{\var(\mu_1, \mu_1)}.
\end{equation}

The $H_n$-center is defined in \cite{Bam20a} associated to a base point along the Ricci flow.

\begin{definition}\label{hnc}
A point $(z, t)\in M\times I$ is called an $H_n$-center of a point $(x_0, t_0)\in M \times I$ if $t< t_0$ and
\begin{equation}
\var_t(\delta_z, \nu_{x_0, t_0; t}) \leq H_n (t_0 - t),
\end{equation}
where $\var_t$ is the variance with respect to the metric $g(t)$.

\end{definition}
Due to \cite[Proposition 3.12]{Bam20a}, given any $(x_0, t_0)$ and $t<t_0$, there exists at least one $H_n$-center of $(x_0, t_0)$. Immediately, if $(z, t)$ is an $H_n$-center of $(x_0, t_0)$, then we have
\begin{equation}\label{dvh}
d^g_{W_1}(\delta_z, \nu_{x_0, t_0; t} ) \leq \sqrt{\var(\delta_z, \nu_{x_0, t_0; t} )} \leq \sqrt{H_n(t_0-t)}.
\end{equation}
The following lemma if proved in \cite{Bam20a}, which asserts that the mass of the conjugate heat kernel measure will concentrate around the $H_n$-centers.
\begin{lemma} \label{hnc'2}
If the point $(z, t)$ is an $H_n$-center of $(x_0, t_0)$ with $t< t_0$, then for any $A>0$, we have
$$\nu_{x_0, t_0; t} \left( B\left( z, t, \sqrt{AH_n (t_0-t)} \right) \right) \geq 1- \frac{1}{A} .$$
\end{lemma}

We now define the Nash entropy introduced by Hein-Naber \cite{HN}. Let $d\nu= (4\pi \tau)^{-n/2} e^{-f} dg$ be a probability measure on a closed $n$-dimensional Riemannian manifold $(M, g)$ with $\tau>0$ and $f\in C^\infty(M)$. The Nash entropy is defined by
\begin{equation}
\cN[g, f, \tau] = \int_M f d\nu - \frac{n}{2}.
\end{equation}

We can rewrite the conjugate heat measure based at $(x_0, t_0)\in M\times I$ by
$$d\nu_{x_0, t_0; t} = K(x_0, t_0; \cdot, t) dg(t) = (4\pi \tau)^{-n/2} e^{-f(t)} dg(t),$$
where $\tau = t_0 - t\geq 0$. Then we define the pointed Nash entropy along Ricci flow based at $(x_0, t_0)$ by 
\begin{equation}
\cN_{x_0, t_0}(\tau) = \cN [ g(t_0-\tau), f(t_0-\tau), \tau].
\end{equation}
Then we set $\cN_{x_0, t_0}(0) =0$, which makes $\cN_{x_0, t_0}(\tau)$ being continuous at $\tau=0$. We also define
\begin{equation}
\cN_s^*(x_0, t_0)= \cN_{x_0, t_0}(t_0-s),
\end{equation}
for $s< t_0$ and $s\in I$. The pointed Nash entropy $\cN_{x_0, t_0}(\tau)$ is non-increasing when $\tau\geq 0$ is increasing.

For any compact, n-dimensional manifold $(M, g)$, Perelman's $\cW$-functional is defined by, for any $\tau > 0$,
$$\cW[g, f, \tau] = (4\pi \tau)^{-n/2} \int_M \left(\tau ( |\nabla f|^2 + R) + f \right) e^{-f} dg,$$
with $f\in C^\infty(M)$ so that $\int_M (4\pi \tau)^{-n/2} e^{-f} dg =1$, and Perelman's $\mu$-functional and $\nu$-functional are defined by
$$\mu[g, \tau]=\underset{\int_M (4\pi \tau)^{-n/2} e^{-f} dg =1}{\inf} \cW[g, f, \tau],$$
and 
$$\nu[g, \tau]=\underset{0<\tau'<\tau}{\inf} \mu[g, \tau'].$$
If $(M, (g_t)_{t\in [0, T)})$ is a Ricci flow, then the functions $t\to \mu[g_t, T-t]$ and $t\to \nu[g_t, T-t]$ are non-decreasing. It is proved in \cite{Bam20a} that 
\begin{equation}\label{NM}
\cN^*_t (x_0, t_0) \geq \mu[g(t), t_0-t],  
\end{equation}
for any $t<t_0$. 
In \cite{Bam20a}, Bamler established systematic results on the Nash entropy and heat kernel bounds on Ricci flow background. 
The following quantitative volume estimate is established in \cite{Bam20a}, extending Perelman's volume non-collapsing estimates.

\begin{lemma} \label{lvnc}
Let $(M, g(t))_{t\in [-r^2, 0]}$ be a solution of the Ricci flow. If $R\leq r^{-2}, ~ on~ B_{g(0)}(x, r)$, then we have
\begin{equation}
\vol_{g(0)} (B_{g(0)}(x, r)) \geq c \exp\left(\cN^*_{-r^2}(x, 0)\right) r^n.
\end{equation}

\end{lemma}

The assumption on the scalar curvature upper bound can be removed near the $H_n$-center as proved in \cite{Bam20a}.
\begin{lemma} \label{hcv}  
Let $(M, g(t))_{t\in [-r^2, 0]}$ be a solution of the Ricci flow. Suppose $(z, -r^2)$ is an $H_n$-center of $(x_0, 0)$ and 
$$R(\cdot, -r^2) \geq R_{min},$$
for some fixed $R_{min} \in \mathbb{R}$. Then there exists $c= c(R_{min} r^2)>0$ such that 
\begin{equation}
\vol_{g(-r^2)}(B_{g(-r^2)}(z, (2H_n)^{1/2} r)) \geq c \exp\left(\cN^*_{-r^2}(x, 0)\right) r^n.
\end{equation}

\end{lemma}

Next, we have the heat kernel upper bound estimate below proved in \cite[Theorem 7.2]{Bam20a}.
\begin{lemma} 
Let $(M, g(t))_{t\in I}$ be a solution of the Ricci flow. Suppose that on $M \times [s, t]$,
$$[s, t]\subset I, ~ R\geq R_{min}.$$
Let $(z, s) \in M\times I$ be an $H_n$-center of $(x, t)\in M\times I$. Then there exist $C=C(R_{min}(t-s))<\infty$, such that for any $y\in M$, we have
\begin{equation}
K(x, t; y,s) \leq C (t-s)^{-n/2} \exp\left( -\cN^*_{s}(x, t) \right) \exp\left( - \frac{ d^2_s(z, y)}{C (t-s)} \right) .
\end{equation}

\end{lemma}
The following estimate relating the $W_1$-distance to the $\mathcal{L}$-length follows from the above heat kernel upper bound estimate and Perelman's Harnack inequality \cite[Lemma 21.2]{Bam20c}. 
\begin{lemma} \label{wdn} 
Let $(M, g(t))_{t\in (-T, 0)}$ be a solution of the Ricci flow for some $T>0$. Suppose $(s, t) \subset (-T, 0)$ with $s>-T+\epsilon>0$ for some $\epsilon>0$.  Let  $\gamma: [0, t-s] \rightarrow M \times (-T, 0)$ be a $C^1$ spacetime curve with 
$$\gamma(\tau) \in M \times \{ t-\tau\}, ~~\gamma(0)=x, ~\gamma(t-s) = y. $$
Then there exists $C=C(\epsilon)>0$ such that
\begin{equation}
d^{g(s)}_{W_1} (\delta_{y, s}, \nu_{x,t; s}) \leq C \left(1+ \frac{\cL(\gamma)}{2(t-s)^{1/2}} -\cN^*_s(x, t)  \right) ^{1/2} (t-s)^{1/2}.
\end{equation}
\end{lemma} 

Finally, let us recall the following result of Perelman in \cite{Per1}.

\begin{lemma} \label{lcenter} 
Let $(M, g(t))_{t\in (-T, 0)}$ be a solution of the Ricci flow for some $T>0$. Suppose $[s, t] \subset (-T, 0)$. Then for any $x\in M$, there exists a point $y\in M$, such that
$$ \ell_{(x, t)} (y, s) \leq \frac{n}{2}.$$
\end{lemma} 

We define the point $(y,s)$ to be an $\ell_n$-center of $(x,t)$ in Lemma \ref{lcenter}.

%%%%%%%%%%%%%%%%%%%%%%%%%%%%%%%%%%%%%%%%%%%%%%%%%%%%%%%%%%%%%%%%%%%%%%%%%%%%%%%%%%%%

\medskip
\subsection{Proof of Theorem \ref{main1}}

We can now prove Theorem \ref{main1}.
\begin{proof}
Throughout this proof, all the constants will at most depend on $n, \omega_0, B$. Let $t_0 \in [1/2, 1)$ be a given time,  assume we have a positive function $v \in C^{1}\left(X\times [t_0-(1-t_0), t_0] \right)$ satisfying  
\begin{equation}\label{main1LY2}
\frac{ |\partial_t v |}{ v } + \frac{|\nabla v |^2}{ v^2 } \leq \frac{B}{1-t},
\end{equation}
on $X\times [t_0 - (1-t_0 ), t_0]$. We need to show that
$$\frac{ \sup_{X} v ( \cdot, t_0 ) }{ \inf_{X} v ( \cdot, t_0 ) } \leq C.$$
Denote by $A>0$ a positive constant to be determined in the course of the proof. We denote by
$$ f = \ln v : X\times [t_0-(1-t_0), t_0]\to \mathbb{R}, $$
which is also a $C^1$-function. By the assumption (\ref{main1LY}), we have
\begin{equation} \label{geohatu3}
|\partial_t f| + |\nabla f|^2  \leq \frac{B}{1-t}, 
\end{equation}
on $X\times [t_0-(1-t_0), t_0]$. Let $N\geq 0$ be the positive integer such that
\begin{equation} \label{loov}
10NA\leq \sup_{X\times \left\{t_0\right\}} f -\inf_{X\times \left\{t_0\right\}} f \leq 10(N+1)A .
\end{equation}
We only need to obtain an upper bound of $N$. 
We denote by $r_0:=(1-t_0)^{1/2}$. We set, for $k=1, 2, \dots, N$,
$$
A_k:=\inf_{X\times \left\{t_0\right\}} f+k\cdot 10A,
$$ 
hence by the continuity of $f$, for each $k=1, 2, \dots, N$, we can find points $x_k\in X$ such that 
\begin{equation}\label{Equ: def of xk}
f(x_k, t_0)=A_k.
\end{equation}
Next, for each $(x_k, t_0)$, due to Lemma \ref{lcenter} (see also \cite[Section 7]{Per1}), we can find an $\ell_{2n}$-center $(y_k, t_0-r_0^2)$ of $(x_k, t_0)$. That is, we have the following reduced length estimate
$$
\ell_{(x_k, t_0)}\left(y_k, t_0-r_0^2\right)\leq n.
$$
By definition, this means that we can find a smooth spacetime curve $\gamma_k:[0, r_0^2]\to X\times [0, 1)$ connecting $(x_k, t_0)$ to $(y_k, t_0-r_0^2)$ such that
$$
\frac{\mathcal{L}(\gamma_k)}{2\sqrt{t_0-(t_0-r_0^2)}}\leq n,
$$
which we can rewrite as
\begin{equation}\label{Equ: upper bound on L length of gamma k}
\mathcal{L}(\gamma_k)\leq 2nr_0.
\end{equation}
For $\tau\in [0, r_i^2]$, we can compute
$$
\frac{d}{d\tau}f(\gamma_k(\tau), t_0-\tau)=\left\langle\nabla f(\gamma_k(\tau), t_0-\tau), \gamma_k'(\tau)\right\rangle_{g(t_0-\tau)}-\partial_t f(\gamma_k(\tau), t_0-\tau).
$$
Hence we can compute
\begin{equation}\label{Equ: difference of u}
\begin{split}
&\left| f(y_k, t_0-r_0^2)- f(x_k, t_0) \right|  \\
=& \left|\int_{0}^{r_0^2}\frac{d}{d\tau} f(\gamma_k(\tau), t_0-\tau)d\tau\right|  \\
%=& \left|\int_{0}^{r_i^2}\left\langle\nabla u(\gamma_k(\tau), t_i-\tau), \gamma_k'(\tau)\right\rangle_{g(t_i-\tau)}-\partial_tu(\gamma_k(\tau), t_i-\tau)d\tau\right|  \\
%
\leq& \left|\int_{0}^{r_0^2}\left\langle\nabla f(\gamma_k(\tau), t_0-\tau), \gamma_k'(\tau)\right\rangle_{g(t_0-\tau)}d\tau\right|+\left|\int_{0}^{r_0^2}\partial_t f(\gamma_k(\tau), t_0-\tau)d\tau\right|.
\end{split}
\end{equation}
For the second term, from the time derivative estimate in (\ref{geohatu3}), we can compute
\begin{equation}\label{Equ: estimate of second term of the difference of u}
\left|\int_{0}^{r_0^2}\partial_t f(\gamma_k(\tau), t_0-\tau)d\tau\right|\leq \int_{0}^{r_0^2} \frac{B}{1-t_0} d\tau=B.
\end{equation}
For the first term, combining the gradient estimate in (\ref{geohatu3}), (\ref{Equ: upper bound on L length of gamma k}), and the fact that $R(t)\geq -C$ for all $t\in [0, 1)$, we can compute
\begin{equation}\label{Equ: estimate of first term of the difference of u}
\begin{split}
&\left|\int_{0}^{r_0^2}\left\langle\nabla f(\gamma_k(\tau), t_0-\tau), \gamma_k'(\tau)\right\rangle_{g(t_0-\tau)}d\tau\right| \\
\leq& \int_{0}^{r_0^2}\left|\nabla f\right|_{g(t_0-\tau)}(\gamma_k(\tau), t_0-\tau)\cdot\left|\gamma_k'(\tau)\right|_{g(t_0-\tau)}d\tau \\
\leq& Cr_0^{-1}\int_{0}^{r_0^2}\left|\gamma_k'(\tau)\right|_{g(t_0-\tau)}d\tau \\
\leq& Cr_0^{-1}\left(\int_{0}^{r_0^2}\tau^{-\frac{1}{2}}d\tau\right)^{\frac{1}{2}}\cdot\left(\int_{0}^{r_0^2}\tau^{\frac{1}{2}}\left|\gamma_k'(\tau)\right|_{g(t_0-\tau)}^2d\tau\right)^{\frac{1}{2}} \\
%\leq& 2Cr_0^{-\frac{1}{2}}\left(\int_{0}^{r_0^2}\tau^{\frac{1}{2}}\left|\gamma_k'(\tau)\right|_{g(t_0-\tau)}^2d\tau\right)^{\frac{1}{2}} \\
%
\leq& Cr_0^{-\frac{1}{2}}\left(\int_{0}^{r_0^2}\tau^{\frac{1}{2}}\left(\left|\gamma_k'(\tau)\right|_{g(t_0-\tau)}^2+R(\gamma_k(\tau), t_0-\tau)+C\right)d\tau\right)^{\frac{1}{2}} \\
\leq& Cr_0^{-\frac{1}{2}}\left(\mathcal{L}(\gamma_k)+\int_{0}^{r_0^2}Cd\tau\right)^{\frac{1}{2}} \\
\leq& Cr_0^{-\frac{1}{2}}\left(2nr_0+C(n)r_0^2\right)^{\frac{1}{2}} \leq C.\\
\end{split}
\end{equation}
Plugging (\ref{Equ: estimate of second term of the difference of u}) and (\ref{Equ: estimate of first term of the difference of u}) into (\ref{Equ: difference of u}), we obtain
\begin{equation}\label{Equ: estimate of difference of u}
\left|f(y_k, t_0-r_0^2) - f(x_k, t_0)\right|\leq C_1,
\end{equation}
for each $k=1, 2, \dots, N$. We then have the following claim.
\begin{claim}\label{claim: lower bound on the distance of yks}
If we choose $A$ large enough so that $A\geq C_1$, then for any $k, l\in\left\{1, 2, \dots, N\right\}$, $k\neq l$, we have
\begin{equation}\label{Equ: lower bound on the distance of yks}
d_{g(t_0-r_0^2)}(y_k, y_l)\geq B^{-\frac{1}{2}}Ar_0.
\end{equation}
\end{claim}
\begin{proof} 
Indeed, assmue for some $k\ne l\in\left\{1, 2, \dots, N\right\}$, we have
$$
d_{g(t_0-r_0^2)}(y_k, y_l)\leq B^{-\frac{1}{2}}Ar_0.
$$
Then, from the gradient estimate in (\ref{geohatu3}), we have
$$
\left|f(y_k, t_0-r_0^2)-f(y_l, t_0-r_0^2)\right|\leq B^{\frac{1}{2}}r_0^{-1}d_{g(t_0-r_0^2)}(y_k, y_l)\leq B^{\frac{1}{2}}r_0^{-1}\cdot B^{-\frac{1}{2}}Ar_0=A.
$$
Hence applying (\ref{Equ: estimate of difference of u}) and triangle inequality, we have
\[
\begin{split}
& \left|f(x_k, t_0) -f(x_l, t_0)\right| \leq  \left|f(x_k, t_0)-f(y_k, t_0-r_0^2)\right| \\
&\qquad\qquad + \left|f(y_k, t_0-r_0^2)-f(y_l, t_0-r_0^2)\right|+\left|f(y_l, t_0-r_0^2)-f(x_l, t_0)\right|\\
&\qquad\qquad \leq C_1+A+C_1\leq 3A . \\
\end{split}
\]
This is impossible since from (\ref{Equ: def of xk}), we have
$$
\left|f(x_k, t_0)-f(x_l, t_0)\right|=|A_k-A_l|=|k-l|\cdot 10A\geq 10A.
$$
This contradiction proves the claim.
\end{proof}
Now, for each $k=1, 2, \dots, N$, choose $(z_k, t_0-r_0^2)$ be the $H_{2n}$-center of $(x_k, t_0)$. Then from the definition of the $H_{2n}$-center, we have
\begin{equation}\label{Equ: estimate of W_1 distance 1}
\begin{split}
&d^{g(t_0-r_0^2)}_{W_1}\left(\nu_{x_k, t_0; t_0-r_0^2}, \delta_{z_k}\right) \leq \sqrt{\var\left(\nu_{x_k, t_0; t_0-r_0^2}, \delta_{z_k}\right)} \\
&\qquad\qquad \leq \sqrt{H_{2n}(t_0-(t_0-r_0^2))} \leq Cr_0.
\end{split}
\end{equation}
Next, from Lemma \ref{wdn}, the fact that we have uniform $\mu$-entropy lower bound, and (\ref{Equ: upper bound on L length of gamma k}), we can estimate
\begin{equation}\label{Equ: estimate of W_1 distance 2}
\begin{split}
&d^{g(t_0-r_0^2)}_{W_1}\left(\nu_{x_k, t_0; t_0-r_0^2}, \delta_{y_k}\right)\\
\leq& C\left(1+\frac{\mathcal{L}(\gamma_k)}{2\sqrt{t_0-(t_0-r_0^2)}}-\cN_{x_k, t_0}(t_0-(t_0-r_0^2))\right)^{\frac{1}{2}}\sqrt{t_0-(t_0-r_0^2)}\\
\leq&  Cr_0.
\end{split}
\end{equation}
Hence by triangle inequality, combining (\ref{Equ: estimate of W_1 distance 1}) and (\ref{Equ: estimate of W_1 distance 2}), we have
\begin{equation}\label{Equ: estimate of W_1 distance 3}
\begin{split}
&d_{g(t_0-r_0^2)}\left(y_k, z_k\right) = d^{g(t_0-r_0^2)}_{W_1}\left(\delta_{y_k}, \delta_{z_k}\right)\\
& \qquad \leq d^{g(t_0-r_0^2)}_{W_1}\left(\delta_{y_k}, \nu_{x_k, t_0; t_0-r_0^2} \right)+d^{g(t_0-r_0^2)}_{W_1}\left(\nu_{x_k, t_0; t_0-r_0^2}, \delta_{z_k}\right)\\
& \qquad \leq  Cr_0.
\end{split}
\end{equation}
Hence, for some constant $C_2<\infty$, we have
\begin{equation}\label{Equ: containment of balls}
B_{g(t_0-r_0^2)}\left(z_k, \sqrt{2H_{2n}}\cdot r_0\right)\subset B_{g(t_0-r_0^2)}\left(y_k, C_2r_0\right),
\end{equation}
Due to Claim \ref{claim: lower bound on the distance of yks}, if we choose $A$ sufficiently large such that 
$$
B^{-\frac{1}{2}}A>2C_2,
$$
then we must have the $N$-balls
$$
\left\{B_{g(t_0-r_0^2)}\left(y_k, C_2r_0\right)\right\}_{k=1,2,\dots,N},
$$
are mutually disjoint, and combining with (\ref{Equ: containment of balls}), we must have the $N$-balls
$$
\left\{B_{g(t_0-r_0^2)}\left(z_k, \sqrt{2H_{2n}}\cdot r_0\right)\right\}_{k=1,2,\dots,N},
$$
are mutually disjoint. But due to Bamler's volume non-collapsing estimate on balls around the $H_{2n}$-center, say Lemma \ref{hcv}, we have
$$
\vol_{g(t_0-r_0^2)}\left(B_{g(t_0-r_0^2)}\left(z_k, \sqrt{2H_{2n}}\cdot r_0\right)\right)\geq c(n)\exp\left(\cN_{x_k, t_0}(r_0^2)\right)r_0^{2n}\geq \kappa r_0^{2n},
$$
for each $k=1, 2, \dots, N$, for some constant $\kappa>0$. Hence we have 
\begin{equation} \label{zball4}
\begin{split}
\vol_{g(0)}\left(X\right) \cdot (2r_0^2)^{n} &= \vol_{g(t_0-r_0^2)} \left( X \right) \\
&\geq \sum_{k=1}^{N} \vol_{g(t_0-r_0^2)}\left( B_{g(t_0-r_0^2)}\left(z_k, \sqrt{2H_{2n}} \cdot r_0\right) \right)  \\
&\geq N \kappa r_0^{2n},
\end{split}
\end{equation}
hence we conclude that $N\leq C$, which completes the proof.
\end{proof}
%

%%%%%%%%%%%%%%%%%%%%%%%%%%%%%%%%%%%%%%%%%%%%%%%%%%%%%%%%%%%%%%%%%%%%%%%%%%%%%%%%%%%%
%%%%%%%%%%%%%%%%%%%%%%%%%%%%%%%%%%%%%%%%%%%%%%%%%%%%%%%%%%%%%%%%%%%%%%%%%%%%%%%%%%%%

\section{The Ricci potential}\label{ripp}

In this section, we will reduce the K\"ahler-Ricci flow to a complex Monge-Amp\`ere flow and define the Ricci potential for the evolving K\"ahler metrics. 

Let $\omega_0$ and $\tilde \omega(s)$ be the K\"ahler forms associated to the K\"ahler metric $g_0$ and $\tilde g(s)$ in (\ref{nkrflow1}). There exists a smooth volume form  $\Omega$  such that
$$-\ddbar \log \Omega = \omega_0$$
since $\omega_0\in c_1(X)$. Then  (\ref{nkrflow1}) can be reduced to the following complex Monge-Amp\`ere flow  
\begin{equation}\label{maflow2}
\left\{
\begin{array}{l}
{ \displaystyle \dds{\varphi} = \log \frac{ (\omega_0 + \ddbar \varphi )^n }{\Omega} + \varphi, ~~ s\in [0, \infty), }\\
\\
\varphi|_{ s=0} = 0,
\end{array} \right. 
\end{equation}
where $\tilde \omega(s)= \omega_0 + \ddbar \varphi(\cdot, s)$.
As in \cite{SeT},  we denote the Ricci potential by 
\begin{equation} \label{defu0}
u_0 = \dds{\varphi}
\end{equation}
for the flow (\ref{nkrflow1}) on $X\times [0, \infty)$.  $u_0$ satisfies the following coupled equations
\begin{equation}\label{ceofu0}
\left\{
\begin{array}{l}
\dds{} u_0 = \Delta u_0 + u_0=n- R_{\tilde g}(s) + u_0 , \\
\\
\ric (\tilde\omega(s)) = \tilde\omega(s) - \ddbar u_0 , 
\end{array} \right. 
\end{equation}
where $\Delta$ is the Laplacian associated to $\tilde g$.  Standard computation gives the evolution equations for $u_0$. 
\begin{equation} \label{eogu0}
\left( \dds{} - \Delta \right) |\nabla u_0|^2 = |\nabla  u_0|^2 - |\nabla\nabla u_0|^2 - |\nabla\overline{\nabla} u_0|^2 ,
\end{equation}
\begin{equation} \label{eolu0}
\left( \dds{} - \Delta \right) \Delta u_0 = \Delta u_0 - |\nabla\overline{\nabla} u_0|^2 .
\end{equation}

%%%%%%%%%%%%%%%%%%%%%%%%%%%%%%%%%%%%%%%%%%%%%%%%%%%%%%%%%%%%%%%%%%%%%%%%%%%%%%%%%%%%
%%%%%%%%%%%%%%%%%%%%%%%%%%%%%%%%%%%%%%%%%%%%%%%%%%%%%%%%%%%%%%%%%%%%%%%%%%%%%%%%%%%%

%%%%%%%%
%%%%%%%%%%%%%%%%%%%%%%%%%%%%%%%%%%%%%%%%%%%%%%%%%%%%%%%%%%%%%%%%%%%%%%%%%%%%%%%%%%%%
 
 \smallskip
 
\section{A new normalization}\label{anewnorm}

In this section, we propose a new normalization for the Ricci potential $u_0$, which is different from Perelman's original normalization (cf. Section \ref{Pernorm}). We then prove the similar gradient and Laplacian estimates as those in Perelman's original proof. In \cite{JST23a}, we extend this normalization to general finite time solution of K\"ahler-Ricci flow on K\"ahler manifolds.
We will work with the normalized flow  (\ref{nkrflow1}) with $(X, \tilde \omega(s))$ and $s\in [0, \infty)$ as well as its corresponding Monge-Amp\`ere flow (\ref{maflow2}).

We start with the Ricci potentail $u_0$ defined in (\ref{defu0})  defined for the normalized flow. For $s\in [0, \infty)$, we denote by
\begin{equation}\label{deofa}
a(s) = \inf_{X} u_0 ( \cdot , s) ,
\end{equation}
which is a continuous function. 
We then define the normalized Ricci potential by 
\begin{equation}\label{deofu}
u = u_0 - a(s) + 1 .
\end{equation}
We will prove the same gradient and Laplace estimates for our new normalization as those of Perelman by his original normalization.
\begin{proposition} \label{gandleou2}
There exists $C=C(n, \omega_0)>0$, such that
\begin{equation} \label{geou2}
\frac{ |\nabla u|^2}{ u } \leq C, 
\end{equation}
\begin{equation} \label{leou2}
\frac{ |\Delta u| }{ u } \leq C. 
\end{equation}
on $X\times [0, \infty)$.
\end{proposition}
\begin{proof}
Throughout this proof, all the constants will depend on $n$, $X$ and $\omega_0$.
We first prove the gradient estimate (\ref{geou2}). Since $u\geq 1$ globally on $X\times [0, \infty)$, we can define 
$$ \cH = \frac{ |\nabla u|^2 }{ u } ,$$ 
which is a smooth function on each time-slice and a continuous function in the time. 

Now assume that for some $\tilde s>0$, $\cH$ achieves its maximum on $X\times [0, \tilde s]$ at the point $(\tilde x, \tilde s)$. To avoid technical complications and clarify the main ideas, we assume that $a(s)$ is differentiable at $s=\tilde s$. When $a(s)$ is not differentiable at $s=\tilde s$, see Remark \ref{aisnotC1}. We first derive the estimate for the time derivative of $a$. 
\begin{claim} \label{FDQofa}
There exists a constant $C=C(n, \omega_0)<\infty$, such that
$$
\partial_s a(\tilde s) \leq a(\tilde s) + C .
$$
%for all $s\in [0, \infty)$.
%
\end{claim} 
\begin{proof}
Let $p_{\tilde s}\in X$ to be the Ricci vertex of $u_0$ at time $\tilde s$. Then we can compute
\begin{equation}
\begin{split}
\lim_{\delta \to 0^+} & \frac{a(\tilde s+\delta )-a(\tilde s)}{\delta } = \lim_{\delta \to 0^+} \frac{a(\tilde s+\delta )-u_0(p_{\tilde s}, \tilde s)}{\delta } \\
& \leq \lim_{\delta \to 0^+} \frac{u_0(p_{\tilde s}, \tilde s+\delta )-u_0(p_{\tilde s}, \tilde s)}{\delta } = (\partial_s u_0)(p_{\tilde s}, \tilde s) \\
& = (n-R_{\tilde g} + u_0)(p_{\tilde s}, \tilde s) \leq u_0(p_{\tilde s}, \tilde s) + C \\
& = a(\tilde s) + C ,
\end{split}
\end{equation}
where the uniform lower bound of the scalar curvature is used.%
\end{proof}

Using (\ref{eogu0}), on $X\times \left\{ \tilde s \right\}$, we can compute 
\begin{equation} \label{eogu1}
\begin{split}
& \left( \partial_s - \Delta \right)  \cH \\
= & \frac{\left( \partial_s - \Delta \right)|\nabla u|^2}{ u }  - \frac{ 2 |\nabla u|^4 }{ u^3 } + \frac{ 2 Re\left\{ \nabla |\nabla u|^2\cdot\overline{\nabla} u\right\}}{ u^2}  - \frac{ |\nabla u|^2 }{ u^2 } \left( \partial_s - \Delta \right) ( u_0 - a ) \\
= & \frac{|\nabla  u|^2 - |\nabla\nabla  u|^2 - |\nabla\overline{\nabla}  u|^2}{ u } - \frac{ 2 |\nabla u|^4 }{u^3} \\
&\qquad\qquad + \frac{ 2 Re\left\{ \nabla |\nabla u|^2\cdot\overline{\nabla} u\right\}}{ u^2}  + \frac{ |\nabla u|^2 }{ u^2 } ( \partial_s a - u_0 ) \\
= & -\frac{|\nabla\nabla  u|^2 + |\nabla\overline{\nabla}  u|^2}{ u } - \frac{ 2 |\nabla u|^4 }{u^3} + \frac{ 2 Re\left\{ \nabla |\nabla u|^2\cdot\overline{\nabla} u\right\}}{ u^2} + \frac{ |\nabla u|^2 }{ u^2 } ( \partial_s a -a +1 ), 
\end{split}
\end{equation}
Hence by Lemma \ref{FDQofa}, on $X\times \left\{ \tilde s \right\}$, we can compute
\begin{equation} \label{eogu2}
\begin{split}
& \left( \partial_s - \Delta \right)  \cH \\
\leq & -\frac{|\nabla\nabla  u|^2 + |\nabla\overline{\nabla}  u|^2}{ u } - \frac{ 2 |\nabla u|^4 }{u^3} + \frac{ 2 Re\left\{ \nabla |\nabla u|^2\cdot\overline{\nabla} u\right\}}{ u^2} + \frac{ C|\nabla u|^2 }{ u^2 } \\
= & -\frac{|\nabla\nabla  u|^2 + |\nabla\overline{\nabla}  u|^2}{ u } - \frac{ 2 \ep |\nabla u|^4 }{u^3} + \frac{ 2(1-\ep) Re\left\{ \nabla \cH \cdot\overline{\nabla} u\right\}}{ u }  \\
& + \frac{ C|\nabla u|^2 }{ u^2 } + \frac{ 2\ep Re\left\{ \nabla |\nabla u|^2\cdot\overline{\nabla} u\right\}}{u^2} ,
\end{split}
\end{equation}
Then by the same argument as Perelman (see \cite{SeT}), if $\ep<1/4$, then we have
$$
\frac{ 2\ep Re\left\{ \nabla |\nabla u|^2\cdot\overline{\nabla} u\right\}}{u^2}\leq  \frac{ \ep |\nabla u|^4 }{u^3} + \ep\frac{|\nabla \nabla u|^2 + |\nabla\overline{\nabla} u|^2}{ u },
$$
hence from (\ref{eogu2}), on $X\times \left\{ \tilde s \right\}$ we have
\begin{equation} \label{eogu3}
\begin{split}
& \left( \partial_s - \Delta \right)  \cH \leq -\frac{1}{2}\frac{ |\nabla\overline{\nabla}  u|^2}{ u } - \frac{ \ep |\nabla u|^4 }{u^3} + \frac{ 2(1-\ep) Re\left\{ \nabla \cH \cdot\overline{\nabla} u\right\}}{ u } + \frac{ C|\nabla u|^2 }{ u^2 } .
\end{split}
\end{equation}
Now, at the point $(\tilde x, \tilde s)$, we have $\nabla \cH=0$, therefore by the maximum principle and (\ref{eogu3}), at $(\tilde x, \tilde s)$ we have
\begin{equation} \label{eogu4}
\begin{split}
&0\leq  \left( \partial_s - \Delta \right)  \cH \leq - \frac{ \ep |\nabla u|^4 }{u^3} + \frac{ C|\nabla u|^2 }{ u^2 } ,
\end{split}
\end{equation}
from which we conclude that
$$
\cH (\tilde x, \tilde s) = \frac{ |\nabla u|^2 }{ u } (\tilde x, \tilde s) \leq C,
$$
This proves (\ref{geou2}).

Next, for the Laplacian estimate (\ref{leou2}), we define
$$
K_1:= \frac{ - \Delta u + C_0 }{ u },\qquad   \cK := K_1 + 100\cH,
$$
where $C_0=C_0(n, \omega_0)<\infty$ is a constant such that
$$
- \Delta u + C_0 = R_{\tilde g}(s) - n + C_0 \geq 0
$$
holds globally. As before, we assume that for some $\tilde s>0$, $\cK$ achieves its maximum on $X\times [0, \tilde s]$ at the point $(\tilde x, \tilde s)$, and that $a(s)$ is differentiable at $s=\tilde s$. Then using (\ref{eolu0}) and Lemma \ref{FDQofa}, on $X\times \left\{ \tilde s \right\}$ we compute
\begin{equation} \label{eolu2}
\begin{split}
& \left( \partial_s - \Delta \right) K_1  \\
& = \frac{  \left( \partial_s - \Delta \right) (- \Delta u) }{ u } - \frac{ - \Delta u + C_0 }{ u^2 } \left( \partial_s - \Delta \right) ( u_0 - a ) + \frac{ 2  Re\left\{ \nabla K_1 \cdot \overline{\nabla} u\right\} }{ u } \\
& = \frac{ |\nabla\overline{\nabla} u|^2 }{ u } + \frac{ 2  Re\left\{ \nabla K_1 \cdot \overline{\nabla} u\right\}  }{ u } + \frac{ -\Delta u + C_0 }{ u^2 } \left( \partial_s a - a + 1 \right) - \frac{ C_0 }{ u } \\
& \leq \frac{ |\nabla\overline{\nabla} u|^2 }{ u } + \frac{ 2  Re\left\{ \nabla K_1 \cdot \overline{\nabla} u\right\}  }{ u } + C\frac{ |\nabla\overline{\nabla} u| + 1 }{ u^2 }  .
\end{split}    
\end{equation}
Combining (\ref{eogu3}) (with $\ep=0$ there) and (\ref{eolu2}), and the gradient estimate (\ref{geou2}), on $X\times \left\{ \tilde s \right\}$, we obtain that
\begin{equation} \label{eolu3}
\begin{split}
\left( \partial_s - \Delta \right) \cK & \leq -40\frac{ |\nabla\overline{\nabla}  u|^2}{ u } + \frac{ 2  Re\left\{ \nabla \cK \cdot \overline{\nabla} u\right\}  }{ u } + C\frac{ |\nabla\overline{\nabla} u| + 1 }{ u^2 } \\
& \leq -30\frac{ |\nabla\overline{\nabla}  u|^2}{ u } + \frac{ 2  Re\left\{ \nabla \cK \cdot \overline{\nabla} u\right\}  }{ u } + C  .
\end{split}    
\end{equation}
Now, at the point $(\tilde x, \tilde s)$, we have $\nabla \cK=0$, therefore by the maximum principle and (\ref{eolu3}), at $(\tilde x, \tilde s)$ we have
\begin{equation} \label{eolu4}
\begin{split}
&0\leq  \left( \partial_s - \Delta \right)  \cK \leq -30\frac{ |\nabla\overline{\nabla}  u|^2}{ u } + C ,
\end{split}
\end{equation}
from which we conclude that
$$
\frac{ |\Delta u|^2 }{ u } (\tilde x, \tilde s) \leq C,
$$
Combining the fact that $u\geq 1$ and the gradient estimate (\ref{geou2}), we conclude that $\cK(\tilde x, \tilde s) \leq C$. This proves (\ref{leou2}), hence completes the proof.
\end{proof}       
\begin{remark}\label{aisnotC1}
In the proof of Proposition \ref{gandleou2}, we assume that $a(s)$ is $C^1$. In general, $a(s)$ is only a Lipschitz function (cf. \cite[Lemma 4.2]{JST23a}), but we can always apply perturbation tricks.   For example,  $a(s)$ is perturbed to a smooth function at any based time in \cite[Lemma 4.3]{JST23a} and the same argument can be carried out here verbatim. We refer the readers to \cite{JST23a} for more details.

%
%The reason for taking this approach is that, in Theorem \ref{main1}, we need to take the time derivative to the normalized Ricci potential. In \cite{JST23a}, such new mechanism is extended to a general Harnack estimate on Ricci flow, and we apply it successfully in the Fano bundle case, which makes Perelman's estimates being the special case of the base space being a single point.
    
\end{remark}

%%%%%%%%%%%%%%%%%%%%%%%%%%%%%%%%%%%%%%%%%%%%%%%%%%%%%%%%%%%%%%%%%%%%%%%%%%%%%%%%%%%%
%%%%%%%%%%%%%%%%%%%%%%%%%%%%%%%%%%%%%%%%%%%%%%%%%%%%%%%%%%%%%%%%%%%%%%%%%%%%%%%%%%%%
 
\section{Upper bound of the Ricci potential}\label{UBofRP}

In this section, we apply Theorem \ref{main1} to prove Theorem \ref{eoper}. Throughout this section, all the constants will at most depend on $n, \omega_0$.

Recall that  $\omega(t)$ with  $t\in [0, 1)$ is the solution of the unnormalized flow (\ref{unkrflow1}) and $\tilde\omega(s)$ with  $s\in [0, \infty)$ is the solution of the normalized  flow (\ref{nkrflow1}). They are related by 
\begin{equation}\label{rb unkrf and nkrf'3}
s=-\ln (1-t),~~ t=1-e^{-s},~~ \tilde{\omega}(s)=(1-t)^{-1}\omega(t).
\end{equation}
Fix a time $t_0\in [1/2, 1)$, and set $s_0=-\ln (1-t_0)$. Let $u$ be the normalized Ricci potential defined in (\ref{deofu}). By Proposition \ref{gandleou2}, we have 
\begin{equation} \label{esofv3}
\frac{| \Delta u  |}{ u } + \frac{|\nabla u |^2}{ u } \leq C,
\end{equation}
on $X\times [0, \infty)$. To obtain the time-derivative estimate, we need the following lemma by elementary Calculus calculations.
\begin{lemma}[Lemma 4.3 in \cite{JST23a}] \label{a'3} 
For any $T\geq 0$, there exists $B=B(n, X, g_0)>0 $ such that for $s\in [0, s_0+T]$ we have
\begin{equation} \label{a'4}
e^{s-s_0}a(s_0)-B\leq a(s).
\end{equation}
\end{lemma}

Given $s_0$, there exists a uniform constant $B>0$ independent of $s_0$ by Lemma \ref{a'3} (with $T=0$)) such that for all $s\in [0, s_0]$, we have
\begin{equation} \label{a3}
e^{s-s_0}a(s_0)-B\leq a(s).
\end{equation}
We define
$$
v=u_0-(e^{s-s_0}a(s_0)-B) + 1 ,
$$
on $X\times [0, s_0]$, which is a smooth function. Then $v\geq 1$ on $X\times [0, s_0]$ and it satisfies the following coupled equations
\begin{equation}\label{ceofv2}
\left\{
\begin{array}{l}
\partial_s v = \Delta v + v - (B+1)=n- R_{\tilde g}(s) + v - (B+1) , \\
\\
\ric (\tilde\omega(s)) = \tilde\omega(s) - \ddbar v ,
\end{array} \right. 
\end{equation}
on $X\times [0, s_0]$  from (\ref{ceofu}). Hence from (\ref{esofv3}), there exists $C>0$ such that 
\begin{equation} \label{esofv4}
\frac{ |\partial_s v |}{ v } + \frac{| \Delta v  |}{ v } + \frac{|\nabla v |^2}{ v } \leq C
\end{equation}
on $X\times [0, s_0]$. As before, we abuse the notation by identifying $v(t)$ with $v(s(t))$ for $s(t)=-\ln (1-t)$. By (\ref{esofv4}), we have 
\begin{equation} \label{esofv5}
\frac{ |\partial_t v |}{ v } + \frac{| \Delta v  |}{ v } + \frac{|\nabla v |^2}{ v } \leq \frac{C}{1-t},
\end{equation}
on $X\times [0, t_0]$. In particular, $ \inf_{X} v(\cdot, t_0)= B+1$ by the definition of $a$.

 Now, we can apply Theorem \ref{main1} and it gives the following Harnack estimate
$$
\frac{ \sup_{X} v ( \cdot, t_0 ) }{ \inf_{X} v ( \cdot, t_0 ) } \leq C .
$$
Since   $ \inf_{X} v( \cdot, t_0 )\leq C$, 
$$ \sup_{X} v( \cdot, t_0 )\leq C.$$ By the Laplacian estimate, we can conclude that the scalar curvature is uniformly bounded from above by $C(1-t_0)^{-1}$ on $X\times \left\{ t_0 \right\}$ along the unnormalized flow. By Perelman's volume non-collapsing estimate, the diameter is uniformly bounded from above by $C(1-t_0)^{1/2}$ on $X\times \left\{ t_0 \right\}$ in the unnormalized flow since ${\rm Vol}_{g(t)}(X) \leq C (1-t)^n$ for all $t\in [0, 1)$. Thee above estimates hold for all  $t_0\in [1/2, 1)$ and we have proved Theorem \ref{eoper}.

\begin{remark} The uniform upper bound for $ \inf_{X} v( \cdot, t_0 )$ is not used in the proof of Theorem \ref{main1}. It is whereas required in Perelman's proof for the uniform diameter bound. We would also like to point out that Theorem \ref{main1} can also be directly applied to prove Theorem \ref{eoper} for the Ricci potential by Perelman's original normalization (\ref{nov}).

\end{remark}

\section{Perelman's normalization for the Ricci potential}\label{LBofPernorm}

In this section, we provide two new arguments to obtain the lower bound of the Ricci potential by Perelman's original normalization.  

%%%%%%%%%%%%%%%%%%%%%%%%%%%%%%%%%%%%%%%%%%%%%%%%%%%%%%%%%%%%%%%%%%%%%%%%%%%%%%%%%%%%
%\smallskip
\subsection{Perelman's normalization}\label{Pernorm}

We first recall Perelman's normalization for the Ricci potential $u_0$ define in (\ref{defu0}). 

First, we find a smooth function $b(s): [0, \infty) \to \mathbb{R}$ such that
\begin{equation}\label{nov}
\int_X e^{-u_0-b} d\tilde g(s) = (2\pi)^{n},
\end{equation}
Then we denote the Ricci potential by 
\begin{equation} \label{defu}
u = u_0 + b(s) .
\end{equation}
If we let $a(s)=b'(s)-b(s)$, then $u$ satisfies the following coupled equations 
\begin{equation}\label{ceofu}
\left\{
\begin{array}{l}
\partial_s u = \Delta u + u + a = n-R_{\tilde g}(s) + u + a , \\
\\
\ric (\tilde\omega(s)) = \tilde\omega(s) - \ddbar u ,
\end{array} \right.
\end{equation}
on $M\times [0, \infty)$. Differentiating (\ref{nov}), we can find that
\begin{equation}\label{foa}
a(s)= -(2\pi)^{-2n}\int_X u e^{- u } d\tilde g(s).
\end{equation}
Using the entropy lower bound, we have
\begin{equation}\label{boa}
|a(s)|\leq C, ~~ s\in [0, \infty),
\end{equation}
for some uniform constant $C=C(n, g_0)<\infty$. Perelman then proved the following important lower bound of the normalized Ricci pontential $u$.

\begin{lemma} \label{lbou}
There exists constant $C=C(n, g_0)<\infty$, such that
$$ u(\cdot, s)  \geq  -C, $$
for all $s\in [0, \infty)$.
\end{lemma} 
This lower bound of $u$ enables Perelman to define and estimate the quantities 
$$
\frac{ |\nabla u|^2}{ u + B }, ~~ \frac{ \left| \Delta u \right|}{ u + B }
$$
for some uniform constant $B=B(n,X, g_0)>>1$.
Our goal  is to provide two alternative proofs for  Lemma \ref{lbou}.
%

%%%%%%%%%%%%%%%%%%%%%%%%%%%%%%%%%%%%%%%%%%%%%%%%%%%%%%%%%%%%%%%%%%%%%%%%%%%%%%%%%%%%
%\medskip
\subsection{Proof of Lemma \ref{lbou} by Moser's iteration}\label{moseriapp}

First, we recall the following Sobolev inequality along Ricci flow due to \cite[Theorem 1.5]{Ye}, see also \cite{ZhQ07}.
\begin{proposition} \label{SobolevaRF}
Let $(M, g(t))$, $t\in [0, T]$ be a solution of the Ricci flow on compact $n$-manifold. Then there exists a constant $C=C(n, g(0), T)<\infty$, such that for any $u\in W^{1,2}(M)$ and $t\in [0, T]$, we have
\begin{equation} \label{SobolevaRF2}
\left( \int_M |u|^{\frac{2n}{n-2}} dg(t) \right)^{\frac{n-2}{n}} \leq C\int_M \left( |\nabla u|^2 + \frac{R}{4} u^2 \right) dg(t) + C\int_M u^2 dg(t), 
\end{equation}
on $X\times [0, T]$.
\end{proposition}
Now we can give a Moser iteration approach to Lemma \ref{lbou}.
\begin{proof}[Proof of Lemma \ref{lbou}]
Throughout the proof, the constants $C$ depends at most on $n, \omega_0$.
First, we apply Proposition \ref{SobolevaRF} to the solution of the unnormalized flow (X, g(t)), $t\in [0,1)$, we obtain that, for any $v\in W^{1,2}(X)$ and $t\in [0, 1)$, we have
\begin{equation} \label{SobolevaRF3}
\left( \int_X |v|^{\frac{2n}{n-1}} dg(t) \right)^{\frac{n-1}{n}} \leq C\int_X \left( |\nabla v|^2 + \frac{R}{4} v^2 \right) dg(t) + C\int_X v^2 dg(t).
\end{equation}
Translating this into the normalized flow $(X, \tilde g(s))$, $s\in [0,\infty)$, we obtain that, for any $v\in W^{1,2}(X)$ and $s\in [0,\infty)$, we have
\begin{equation} \label{SobolevaRF4}
\left( \int_X |v|^{\frac{2n}{n-1}} d\tilde g(s) \right)^{\frac{n-1}{n}} \leq C\int_X \left( |\nabla v|^2 + \frac{R_{\tilde g}(s)}{4} v^2 \right) d\tilde g(s) + Ce^{-s}\int_X v^2 d\tilde g(s).
\end{equation}
Taking trace of the second equation in (\ref{ceofu}), we have
$$
\Delta u = - R_{\tilde g}(s) + n \leq C,
$$
For any $p\geq 1$, we apply (\ref{SobolevaRF4}) to $v=e^{-\frac{p}{2}u}$, we obtain
\begin{equation} \label{lbofDeltau}
\begin{split}
\int_X & (\Delta u) e^{-pu} d\tilde g(s) = \int_X \frac{4}{p} \left| \nabla e^{-\frac{p}{2}u} \right|^2 d\tilde g(s) \\
& \geq \frac{1}{Cp}  \left\| e^{-\frac{p}{2}u} \right\|_{L^{\frac{2n}{n-1}}}^2 - \frac{1}{p}\int_X R_{\tilde g}(s) e^{-pu} d\tilde g(s) - \frac{4}{p} \int_X e^{-pu} d\tilde g(s) \\
& = \frac{1}{Cp}  \left\| e^{-\frac{p}{2}u} \right\|_{L^{\frac{2n}{n-1}}}^2 + \frac{1}{p}\int_X (\Delta u - n) e^{-pu} d\tilde g(s) - \frac{4}{p}\int_X e^{-pu} d\tilde g(s),
\end{split}
\end{equation}
hence we have
\begin{equation} \label{lbofDeltau2}
\begin{split}
\frac{1}{Cp}  \left\| e^{-\frac{p}{2}u} \right\|_{L^{\frac{2n}{n-1}}}^2 & \leq \left(1-\frac{1}{p}\right)\int_X (\Delta u) e^{-pu} d\tilde g(s) + C\int_X e^{-pu} d\tilde g(s)\\
& \leq C\int_X e^{-pu} d\tilde g(s).
\end{split}
\end{equation}
We set $\beta=\frac{n}{n-1}$, then (\ref{lbofDeltau2}) can be rewriten as 
\begin{equation} \label{lbofDeltau3}
\begin{split}
\left\| e^{-u} \right\|_{L^{p\beta}} \leq \left(Cp\right)^{\frac{1}{p}}\left\| e^{-u} \right\|_{L^{p}}.
\end{split}
\end{equation}
Hence if we denote by $p_k=\beta^k$, then iterating (\ref{lbofDeltau3}) gives 
\begin{equation} \label{lbofDeltau4}
\begin{split}
\left\| e^{-u} \right\|_{L^{p_{k+1}}} & \leq C^{\sum_{j=1}^k\beta^{-j}} \beta^{\sum_{j=1}^kj\beta^{-j}} \left\| e^{-u} \right\|_{L^{1}} \leq C\left\| e^{-u} \right\|_{L^{1}}.  \\
\end{split}
\end{equation}
Letting $k\to\infty$ yields that
$$
\sup_{X} e^{-u} \leq C \left\| e^{-u} \right\|_{L^{1}} = C,
$$
which completes the proof.
\end{proof}
%

%%%%%%%%%%%%%%%%%%%%%%%%%%%%%%%%%%%%%%%%%%%%%%%%%%%%%%%%%%%%%%%%%%%%%%%%%%%%%%%%%%%%
%
\smallskip
\subsection{Proof of Lemma \ref{lbou} using Complex Monge-Amp\`ere equations} \label{CMAeapp}

\begin{proof}[Proof of Lemma \ref{lbou}]
The first half part of the proof is the same to Perelman's proof, we include the details for completeness.
We assume the estimate is not true. Then there exists a time $s_0>0$ and a subset $U\subset X$, such that
\begin{equation}\label{nupofu1}
u(\cdot, s_0)\leq -3A, ~~on ~~ U,
\end{equation}
for some large constant $A=A(n, g_0)>>1$ to be determined. From the first equation in (\ref{ceofu}), we have
$$
\partial_s u = n- R_{\tilde g}(s) + u \leq u + C,
$$
for some constant $C=C(n, g_0)<\infty$, hence from (\ref{nupofu1}) we have
\begin{equation}\label{nupofu2}
u(\cdot, s)\leq e^{s-s_0}(u(\cdot, s_0)+C)\leq -2Ae^{s-s_0}, ~~on ~~ U,
\end{equation}
for all $s\geq s_0$, if we choose $A$ large enough. Then we compute
$$
\partial_s (\varphi + b(s)) = \dds{\varphi} + b'(s)= u+a(s),
$$
then from $a\leq C$, we have 
\begin{equation}\label{dbofvarphi}
\partial_s (\varphi + b(s))\leq u + C,
\end{equation}
for some  $C=C(n, g_0)>0$. Hence for any $x\in U$, we have
\begin{equation}\label{nubofvarphi}
\begin{split}
\varphi(x,s) + b(s) & \leq \varphi(x,s_0) + b(s_0) + \int_{s_0}^{s} (-2Ae^{\tilde s-s_0}+C) d\tilde s\\
& \leq -Ae^{ s-s_0}+C(s_0) ,
\end{split}
\end{equation}
for all $s\geq s_0$, if we choose $A=A(n, g_0)$ large enough, where $C(s_0)>0$ depends on $n, \omega_0, s_0$ but not on the choice of $s\geq s_0$.

Now, using the complex Monge-Amp\`ere flow associated to the K\"ahler-Ricci flow, we have  
$$
e^{-\dds{\varphi}}\tilde\omega(s)^n = e^{-\varphi}\Omega ,
$$
and by (\ref{nubofvarphi}), we have 
\begin{equation}
\begin{split}
(2\pi)^{2n} & = \int_X e^{-u} \tilde\omega(s)^n = \int_X e^{-b(s)} e^{-\dds{\varphi}} \tilde\omega(s)^n \\
& = \int_X e^{-b(s)} e^{-\varphi} \Omega \geq \int_U e^{-(\varphi+b(s))} \Omega \\
& \geq e^{Ae^{ s-s_0} - C(s_0)} \vol_{\Omega}(U).
\end{split}
\end{equation}
Hence 
$$
Ae^{ s-s_0} \leq C(s_0),
$$
for all $s\geq s_0$, which is impossible if we let $s\to\infty$. This completes the proof.
\end{proof}
%

%%%%%%%%%%%%%%%%%%%%%%%%%%%%%%%%%%%%%%%%%%%%%%%%%%%%%%%%%%%%%%%%%%%%%%%%%%%%

%%%%%%%%%%%%%%%%%%%%%%%%%%%%%%%%%%%%%%%%%%%%%%%%%%%%%%%%%%%%%%%%%%%%%%%%%%%%%%%%%%%%

\bigskip
\bigskip

\noindent{\bf Acknowledgements} The authors would like to thank Richard Bamler, Bin Guo, Maximilien Hallgren, Yalong Shi and Zhenlei Zhang for many inspiring discussions. The first named author thanks Xiaochun Rong, Zhenlei Zhang and Kewei Zhang for hospitality and providing an excellent environment during his visits to Capital Normal University and Beijing Normal University where part of this work was carried out.  

\bigskip
\bigskip

%%%%%%%%%%%%%%%%%%%%%%%%%%%%%%%%%%%%%%%%%%%%%%%%%%%%%%%%%%%%%%%%%%%%%%%%%%%%%%%%%%%%


\begin{thebibliography}{99}

\bibitem{Bam18} Bamler, R. {\em Convergence of Ricci flows with bounded scalar curvature}, Ann. of Math. (2) 188 (2018), no. 3, 753--831

\bibitem{Bam20a} Bamler, R. {\em Entropy and heat kernel bounds on a Ricci flow background},arXiv:2008.07093

%\bibitem[Bam20b]{Bam20b} Bamler, R. {\em Compactness theory of the space of Super Ricci flows}, arXiv:2008.09298

\bibitem{Bam20c} Bamler, R. {\em Structure theory of non-collapsed limits of Ricci flows}, arXiv:2009.03243

\bibitem{BZ17} Bamler, R., and Zhang, Q. {\em Heat kernel and curvature bounds in Ricci flows with bounded scalar curvature}, Adv.Math. {\bf 319} (2017), 396-450

%\bibitem{Cao} Cao, H.D. {\em Deformation of K\"ahler metrics to K\"aher–Einstein metrics on compact K\"ahler manifolds}, Invent. Math. 81 (1985), 359--372

\bibitem{CW3} Chen, X.X. and Wang, B. {\em Space of Ricci flows (II)—part B: weak compactness of the flows}, J. Differential Geom, 116 (2020), no. 1, 1--123

\bibitem{Ham} Hamilton, R. {\em Three-manifolds with positive Ricci curvature}, J. Differential Geom, 1982, 17: 255--306

\bibitem{HJST} Hallgren, M., Jian, W., Song, J. and Tian, G. {\em Geometric regularity of the blow-up limits of the K\"ahler-Ricci flow}, preprint

\bibitem{HN} Hein, H. and Naber, A. {\em New logarithmic Sobolev inequalities and an $\varepsilon$-regularity theorem for the Ricci flow}, Comm. Pure Appl. Math. 67 (2014), no. 9, 1543--1561

\bibitem{J} Jian, W. {\em On the improved no-local-collapsing theorem of Ricci flow}, Peking Math. J. 6 (2023), no. 2, 459--468

\bibitem{JS} Jian, W. and Song, J. {\em Diameter estimates for long-time solutions of the K\"ahler-Ricci flow}, Geom. Funct. Anal. 32 (2022), no. 6, 1335--1356

\bibitem{JST23a} Jian, W., Song, J. and Tian, G. {\em Finite time singularities of the K\"ahler-Ricci flow}, preprint


\bibitem{LY} Li, P. and Yau, S.T. {\em On the parabolic kernel of the Scr$\ddot{o}$dinger operator}, Acta Math. {\bf 156} (1986), 153--201

\bibitem{Per1} Perelman, G. {\em The entropy formula for the Ricci flow and its geometric applications}, preprint, math.DG/0211159



\bibitem{PS} Phong, D.H. and Sturm, J. {\em On stability and the convergence of the K\"ahler-Ricci flow}, J. Differential Geom. 72 (2006), no. 1, 149--168


 
\bibitem{SeT} Sesum, N. and Tian, G. {\em Bounding scalar curvature and diameter along the K\"ahler Ricci flow (after Perelman)},  J. Inst. Math. Jussieu 7 (2008), no. 3, 575--587

\bibitem{ST} Song, J. and Tian, G. {\em The K\"ahler-Ricci flow on surfaces of positive Kodaira dimension}, Invent. Math. {\bf 170} (2007), no. 3, 609--653

\bibitem{ST2} Song, J. and Tian, G. {\em Canonical measures and K\"ahler-Ricci flow}, J. Amer. Math. Soc. {\bf 25} (2012), no. 2, 303--353

\bibitem{ST4} Song, J. and Tian, G. {\em The K\"ahler-Ricci flow through singularities}, Invent. Math. {\bf 207} (2017), 519-595

\bibitem{TiZXH2} Tian, G. and Zhu, X.H. {\em  Convergence of the K\"ahler-Ricci flow}, J. Amer. Math. Sci. 17 (2006), 675--699


\bibitem{TiZXH1} Tian, G. and Zhu, X.H. {\em  Convergence of the K\"ahler-Ricci flow on Fano manifolds}, J. Reine Angew. Math. 678 (2013), 223--245


\bibitem{TiZZL} Tian, G. and Zhang, Z.L. {\em Regularity of K\"ahler-Ricci flows on Fano manifolds}, Acta Math. 216 no. 1 (2016), 127--176

\bibitem{Ye} Ye, R. {\em The logarithmic Sobolev and Sobolev inequalities along the Ricci flow}, Commun. Math. Stat. 3 (2015), no. 1, 1--36

\bibitem{ZhQ06} Zhang, Q. {\em Some gradient estimates for the heat equation on domains and for an equation by Perelman}, Int. Math. Res. Not., 39 pages Art. ID 92314, 39, 2006

\bibitem{ZhQ07} Zhang, Q.  {\em Erratum to: A uniform Sobolev inequality under Ricci flow},  Int. Math. Res. Not. 19 (2007)

\bibitem{ZhQ11} Zhang, Q. {\em Bounds on volume growth of geodesic balls under Ricci flow}, Math. Res. Lett. 19 (2012), no. 1, 245-253

\end{thebibliography}
\end{document}